\documentclass[11pt]{amsart}
\usepackage{epsfig}
\usepackage{graphicx}
\usepackage{amscd}
\usepackage{amsmath}
\usepackage{amsxtra}
\usepackage{amsfonts}
\usepackage{amssymb}

\newtheorem{theorem}{Theorem}[section]
\newtheorem{corollary}[theorem]{Corollary}
\newtheorem{lemma}[theorem]{Lemma}
\newtheorem{proposition}[theorem]{Proposition}

\theoremstyle{definition}
\newtheorem{definition}[theorem]{Definition}
\newtheorem{remark}[theorem]{Remark}

\newtheorem{example}[theorem]{Example}
\theoremstyle{remark}

\renewcommand{\theclaim}{\textup{\theclaim}}

\numberwithin{equation}{section}

\def\openone

{\mathchoice

{\hbox{\upshape \small1\kern-3.3pt\normalsize1}}

{\hbox{\upshape \small1\kern-3.3pt\normalsize1}}

{\hbox{\upshape \tiny1\kern-2.3pt\SMALL1}}

{\hbox{\upshape \Tiny1\kern-2pt\tiny1}}}

\makeatletter

\newbox\ipbox

\newcommand{\ip}[2]{\left\langle #1\, , \,#2\right\rangle}
\newcommand{\diracb}[1]{\left\langle #1\mathrel{\mathchoice

{\setbox\ipbox=\hbox{$\displaystyle \left\langle\mathstrut
#1\right.$}

\vrule height\ht\ipbox width0.25pt depth\dp\ipbox}

{\setbox\ipbox=\hbox{$\textstyle \left\langle\mathstrut
#1\right.$}

\vrule height\ht\ipbox width0.25pt depth\dp\ipbox}

{\setbox\ipbox=\hbox{$\scriptstyle \left\langle\mathstrut
#1\right.$}

\vrule height\ht\ipbox width0.25pt depth\dp\ipbox}

{\setbox\ipbox=\hbox{$\scriptscriptstyle \left\langle\mathstrut
#1\right.$}

\vrule height\ht\ipbox width0.25pt depth\dp\ipbox}

}\right. }

\newcommand{\dirack}[1]{\left. \mathrel{\mathchoice

{\setbox\ipbox=\hbox{$\displaystyle \left.\mathstrut
#1\right\rangle$}

\vrule height\ht\ipbox width0.25pt depth\dp\ipbox}

{\setbox\ipbox=\hbox{$\textstyle \left.\mathstrut
#1\right\rangle$}

\vrule height\ht\ipbox width0.25pt depth\dp\ipbox}

{\setbox\ipbox=\hbox{$\scriptstyle \left.\mathstrut
#1\right\rangle$}

\vrule height\ht\ipbox width0.25pt depth\dp\ipbox}

{\setbox\ipbox=\hbox{$\scriptscriptstyle \left.\mathstrut
#1\right\rangle$}

\vrule height\ht\ipbox width0.25pt depth\dp\ipbox}

} #1\right\rangle}

\newcommand{\bz}{\mathbb{Z}}
\newcommand{\M}{\mathcal{M}}

\newcommand{\br}{\mathbb{R}}
\newcommand{\bc}{\mathbb{C}}

\newcommand{\bn}{\mathbb{N}}

\def\blfootnote{\xdef\@thefnmark{}\@footnotetext}

\renewcommand{\mod}{\operatorname{mod}}

\input xy
\xyoption{all}
\usepackage{amssymb}

\def\T{\mathbb{T}}

\def\-{^{-1}}

\def\T{\mathcal{T}}

\def\S{\mathcal{S}}

\def\Rs{{R^T}}
\def\Rsi{(R^T)^{-1}}

\begin{document}
\title[Iterative approximations of exponential bases on fractal measures]{Iterative approximations of exponential bases on fractal measures}
\author{Dorin Ervin Dutkay}
\blfootnote{}
\address{[Dorin Ervin Dutkay] University of Central Florida\\
    Department of Mathematics\\
    4000 Central Florida Blvd.\\
    P.O. Box 161364\\
    Orlando, FL 32816-1364\\
U.S.A.\\} \email{ddutkay@mail.ucf.edu}

\author{Deguang Han}
\address{[Deguang Han]University of Central Florida\\
    Department of Mathematics\\
    4000 Central Florida Blvd.\\
    P.O. Box 161364\\
    Orlando, FL 32816-1364\\
U.S.A.\\} \email{deguang.han@.ucf.edu}

\author{Eric Weber}
\address{[Eric Weber]Department of Mathematics\\
396 Carver Hall\\
Iowa State University\\
Ames, IA 50011\\
U.S.A.\\} \email{esweber@iastate.edu}

\thanks{This work is partially supported by the NSF grant DMS-1106934.}

\begin{abstract}
For some fractal measures it is a very difficult problem in
general to prove the existence of spectrum (respectively, frame,
Riesz and Bessel spectrum). In fact there are examples of
extremely sparse sets that are not even Bessel spectra. In this
paper we investigate this problem for general fractal measures
induced by iterated function systems (IFS). We prove some
existence results of spectra associated with Hadamard pairs. We
also obtain some characterizations of Bessel spectrum in terms of
finite matrices for affine IFS measures, and one sufficient
condition of frame spectrum in the case that the affine IFS has no
overlap.
\end{abstract}
\maketitle \tableofcontents
\section{Introduction}\label{intr}

Joseph Fourier introduced Fourier series on the interval $[0,1)$
(or $[-\pi, \pi)$), i.e. expansions of functions as a series with
terms $e^{2 \pi i n x}$.  These exponentials are now understood as
an orthonormal basis for the Hilbert space $L^{2}[0,1]$ with
respect to Lebesgue measure.  If one considers a Borel probability
measure $\mu$ on $[0,1]$ (or $\br^d$ more generally) other than
Lebesgue measure, a natural question is whether there exists an
orthonormal basis for $L^2(\mu)$ of the form $e^{2 \pi i
\lambda_{n} x}$, for some sequence of frequencies $\{ \lambda_{n}
\} \subset \br$.  If $\mu$ does have such a Fourier basis, it is
called a \emph{spectral measure}.  One of the first examples of a
singular measure which is spectral was given by Jorgensen and
Pedersen in \cite{JP98} which was a measure on a Cantor like set.
Also in \cite{JP98} is a proof that on the usual Cantor middle
third set, the natural measure does not have a
(orthogonal) Fourier basis.

Since this measure does not have an orthonormal basis of
exponentials, we then consider whether this measure has a
non-orthogonal basis of exponentials, such as a Riesz basis
\cite{PW34a}, Schauder basis \cite{Sing70} or even a frame of
exponentials  introduced by Duffin and Schaeffer \cite{DS52a} in
the context of nonharmonic Fourier series. For  the Hilbert space
$L^{2}[0,1]$ with respect to Lebesgue measure, the main result of
Duffin and Schaeffer is a sufficient density condition for
$\{e^{2\pi i\lambda x}\}_{\Lambda\in \Lambda}$ to be a frame.  
An almost complete characterization of the frame properties of $\{e^{2\pi
i\lambda x}\}_{\lambda\in \Lambda}$ in terms of lower Beurling
density was obtained by Jaffard, \cite{Jaffard} and Seip \cite{Seip2}, 
with a complete characterization given in \cite{OSANN} in terms of the
zero sequence of certain Hermite-Biehler functions.  
Moreover, in this classical case, the Bessel sequences
correspond the sets $\Lambda$ that have finite upper Beurling
density, and $\{e^{2\pi i\lambda x}\}_{\lambda\in \Lambda}$  is a
Riesz sequences if  $\Lambda$ is a sparse enough in the sense that
its upper Beurling density is strictly less than $1$. Therefore it
is very easy to construct, Bessel sequences/frames/Riesz sequences
of exponentials for the unit interval. However this is not the
case anymore for fractal measures. For some fractal measures it is
difficult even to prove the existence of frames or Riesz bases.
For example,  it is still an open problem whether there exists a
frame or Riesz basis for the fractal measure on the Cantor middle
third set.  In fact for this measure  there exists a set $\Lambda$
which is extremely sparse but is not a Bessel spectrum \cite{DHSW11}: For any integer
$a\in\bz\setminus\{0\}$, and any infinite set of non-negative
integers $F$, the set $\{3^na\,|\,n\in F\}$ cannot be a Bessel
spectrum for the middle-third Cantor measure $\mu_3$. Moreover,
any such set has upper Beurling dimension $0$.  For this reason it
seems that it is not even clear how to construct a infinite
sequence $\Lambda$ such that its corresponding exponentials form a
Bessel sequence or Riesz sequence. We investigate this problem in
this paper for general fractal measures induced by iterated
function systems (IFS).

\begin{definition}
A sequence $\{x_n\}_{n=1}^{\infty}$ in a Hilbert space (with inner
product $\langle \cdot , \cdot \rangle $) is \emph{Bessel} if
there exists a positive constant $B$ such that for all $v$
\[ \sum_{n=1}^{\infty} | \langle v , x_n \rangle |^2 \leq B \|v\|^2. \]
This is equivalent to the existence of a positive constant $D$
such that for every finite sequence $\{c_{1}, \dots , c_{K} \}$ of
complex numbers
\[ \| \sum_{n=1}^{K} c_{n} x_n \| \leq D \sqrt{\sum_{n=1}^{K} |c_{n}|^2}. \]
Here $D^2 = B$ is called the Bessel bound.

The sequence is a frame if in addition to being a Bessel sequence
there exists a positive constant $A$ such that for all $v$
\[ A \| v\|^2 \leq \sum_{n=1}^{\infty} | \langle v , x_n \rangle |^2 \leq B \|v\|^2. \]
In this case, $A$ and $B$ are called the lower and upper frame
bounds, respectively.

The sequence is a Riesz basic sequence if in addition to being a
Bessel sequence there exists a positive constant $C$ such that for
every finite sequence $\{c_{1}, \dots , c_{K} \}$ of complex
numbers
\[ C \sqrt{\sum_{n=1}^{K} |c_{n}|^2} \leq \| \sum_{n=1}^{K} c_{n} x_n \| \leq D \sqrt{\sum_{n=1}^{K} |c_{n}|^2}. \]
Here $C$ and $D$ are called the lower and upper basis bounds,
respectively.

A sequence $\{x_n\}$ in a Banach space is a Schauder basic
sequence if there exists a constant $E$ such that for every finite
sequence $\{c_{1}, \dots , c_{K} \}$ of complex numbers,
\[ \left\| \sum_{n=1}^{k} c_{n} x_n \right\| \leq E \left\| \sum_{n=1}^{K} c_{n} x_n \right\| \]
where $1 \leq k \leq K$.  $E$ is called the basis constant.
\end{definition}

In what follows we denote by
$$e_\lambda(x):=e^{2\pi i\lambda\cdot x},\quad(x\in \br^d)$$
for $\lambda\in\br^d$. We will also denote by $\widehat\mu$, the
Fourier transform of a measure $\mu$ on $\br^d$:
$$\widehat\mu(x)=\int e^{2\pi ix\cdot t}\,d\mu(t),\quad(x\in\br^d)$$

\begin{definition}\label{defspec} A Borel probability measure $\mu$ on $\br^d$ is called spectral 
if there exists a subset $\Lambda$ of $\br^d$ such that
$E(\Lambda):=\{e_\lambda : \lambda\in\Lambda\}$ is an orthonormal
basis for $L^2(\mu)$. In this case, $\Lambda$ is called a {\it
spectrum} for the measure $\mu$, and $(\mu,\Lambda)$ is called a
{\it spectral pair}.

A subset $\Lambda$ of $\br^d$ is said to be {\it orthogonal} in
$L^2(\mu)$ if the corresponding set of exponential functions
$E(\Lambda)$ is orthogonal in $L^2(\mu)$. The set $\Lambda$ is
called a {\it Bessel/frame/Riesz basic/Schauder basic spectrum} if
the set of exponentials $E(\Lambda)$ is a Bessel
sequence/frame/Riesz basic/Schauder basic sequence. We say that
the corresponding Bessel/frame/Riesz basis bounds are the bounds
for $\Lambda$.
\end{definition}

We will sometimes identify the set $\Lambda$ with the
corresponding set of exponential functions $E(\Lambda)$; for
example, we will say that $\Lambda$ spans a subspace if
$E(\Lambda)$ spans it.

%
%
%

\begin{definition}\label{defaifs}
Consider a $d\times d$ expansive integer matrix $R$ and a subset
$B\subset\bz^d$ with $\#B=N\geq2$. We define the iterated function
system by
$$\tau_b(x)=R^{-1}(x+b),\quad(x\in\br^d).$$

We also define the following operator $\mathcal T$ on Borel
probability measures on $\br^d$
\begin{equation}\label{eqtmu1}
(\T\mu)(E)=\frac1N\sum_{b\in B}\mu(\tau_b^{-1}(E)),
\end{equation}
for any Borel probability measure $\mu$ and all Borel sets $E$.
Equivalently the measure $\T\mu$ is defined by
\begin{equation}
\int f\,d\T\mu=\frac1N\sum_{b\in B}\int f\circ\tau_b\,d\mu,
\label{eqtmu2}
\end{equation}
for all continuous functions $f$ on $\br^d$.

We denote by $\mu_{B}$ the unique invariant measure for the
operator $\T$, i.e., $\T\mu_{B}=\mu_{B}$.  (See \cite{Hut81} for
the existence and uniqueness of this measure)

We say that the measure $\mu_B$ has no overlap if
$$\mu_B(\tau_b(X_B)\cap\tau_{b'}(X_B))=0\mbox{ for all $b\neq b'$ in $B$}.$$

For a subset $L$ of $\br^d$, define the operator $\S$ by
$$\S\Lambda=\bigcup_{l\in L}(R^T\Lambda+l).$$

\end{definition}

\section{Orthogonal exponentials}\label{spec}

We present here a technique for constructing an orthogonal sequence of exponentials 
for an invariant measure for an IFS via a sequence of orthogonal exponentials for 
a sequence of approximating measures.  What is required is that the IFS possesses a 
dual IFS in the sense of Hadamard pairs.  We begin with an orthogonal set of exponentials 
$\Lambda_{0}$ for an initial measure $\mu_0$, and the let the IFS act on $\mu_0$ and the dual IFS 
act on $\Lambda_{0}$, then consider the limit of this process in a specific manner.  
For dimension $1$, in the limit we obtain that the invariant measure for the IFS is a spectral 
measure and construct explicitly a spectrum for that measure (Theorem \ref{th4}).

\begin{definition}\label{defhada}
Let $L$ be a subset of $\br^d$ with $\#L=\#B=N$. We say that
$(B,L)$ is a Hadamard pair if the matrix
$$\frac{1}{\sqrt{N}}\left(e^{2\pi  iR^{-1}b\cdot l}\right)_{b\in B, l\in L}$$
is unitary.

\end{definition}

\begin{proposition}\label{cor2}
Let $(B,L)$ be a Hadamard pair. Define the set
$$\Pi(B)=\left\{ \gamma\in\br^d : \gamma\cdot b\in\bz\mbox{ for all }b\in B\right\}.$$

If $\mu$ is a spectral measure with spectrum $\Lambda$ contained
in $\Pi(B)$, then $\T\mu$ is a  spectral measure with spectrum
$\S\Lambda$.
\end{proposition}

We need some lemmas:

\begin{lemma}\cite{DJ06b}\label{lem1}
Let $\mu_{B}$ be the invariant measure for the IFS $(\tau_b)_{b\in
B}$, i.e., $\T\mu_{B}=\mu_{B}$. Then
\begin{equation}\label{eqscal}
\widehat\mu_{B}(x)=m_{B}(\Rsi x)\widehat\mu_{B}(\Rsi x),\quad(x\in
\br^d)
\end{equation}
where
\begin{equation}
m_{B}(x)=\frac1N\sum_{b\in B}e^{2\pi ib\cdot x}\,(x\in\br^d)
\label{eqchib}
\end{equation}

A set $\Lambda$ is a spectrum for a Borel probability measure
$\mu$ iff
\begin{equation}
\sum_{\lambda\in\Lambda}|\widehat\mu(x+\lambda)|^2=1,\quad(x\in\br^d)
\label{eqortho}
\end{equation}
\end{lemma}

\begin{definition}\label{defh}
For a measure $\mu$ on $\br^d$ and a subset $\Lambda$ of $\br^d$,
define
$$h_{\mu,\Lambda}(x)=\sum_{\lambda\in\Lambda}|\widehat\mu(x+\lambda)|^2,\quad(x\in\br^d)$$

For a pair of subsets $(B,L)$ of $\bz^d$ we define the {\it
transfer operator} on functions on $\br^d$:
$$(R_{B,L}f)(x)=\sum_{l\in L}|m_{B}(\Rsi(x+l))|^2f(\Rsi(x+l)),\quad(x\in\br^d).$$
\end{definition}

\begin{lemma}\label{pr1} Let $(B,L)$ be a Hadamard pair.
\begin{enumerate}
    \item For any measure $\mu$, we have
    $$\widehat{(\T\mu)}(x)=m_{B}(\Rsi x)\widehat\mu(\Rsi x),\quad(x\in\br^d)$$
    \item For any measure $\mu$ and any subset $\Lambda$ of
    $\Pi(B)$, we have
    $$R_{B,L}h_{\mu,\Lambda}=h_{\T\mu,\S\Lambda}.$$
\end{enumerate}
\end{lemma}

\begin{proof}
(i) follows from \eqref{eqtmu2} by taking the Fourier transform in the definition
of $\T\mu$.

For (ii) we compute $h_{\T\mu,\S\Lambda}$. Since $(B,L)$ form a
Hadamard pair it is easy to see that the elements of $L$ are
incongruent $\mod \Rs\Pi(B)$. Therefore, if
$\Lambda\subset\Pi(B)$, then the sets $\Rs\Lambda+l$, $l\in L$ are
disjoint. We then have, using (i):
$$h_{\T\mu,S\Lambda}(x)=\sum_{l\in L}\sum_{\lambda\in\Lambda}|\widehat{\T\mu}(x+\Rs\lambda+l)|^2=
\sum_{l\in
L}\sum_{\lambda\in\Lambda}|m_{B}\left(\Rsi(x+l)+\lambda\right)|^2\left|\widehat\mu\left(\Rsi(x+l)+\lambda\right)\right|^2$$
$$=R_{B,L}h_{\mu,\Lambda}(x),$$
where we used the fact that $m_{B}(y+\lambda)=m_{B}(y)$ for
$\lambda\in\Pi(B)$.

\end{proof}

\begin{proof}[Proof of Proposition \ref{cor2}]
From Lemma \ref{lem1} we know that we have to check that
$h_{\T\mu,S\Lambda}=1$, and we know that $h_{\mu,\Lambda}=1$. But
from Lemma \ref{pr1} we have
$$h_{\T\mu,S\Lambda}=R_{B,L}h_{\mu,\Lambda}=R_{B,L}1=1.$$
In the last equality we used the fact that $(B,L)$ is a Hadamard
pair, which implies $R_{B,L}1=1$
(see \cite{DJ06b}).
\end{proof}

\begin{corollary}\label{cor3}
Let $(B,L)$ be a Hadamard pair. Define the operator $\M$ on
subsets $K$ of $\br^d$ by
\begin{equation}
\M K=\bigcup_{b\in B}\tau_b(K). \label{eqm}
\end{equation}

Let $Q=[0,1)^d$ be the unit cube, and let $\mu_0$ be the Lebesgue
measure on $Q$. Then the measure $\T^n\mu_0$ is the Lebesgue
measure on the disjoint union
$$\M^n Q=\bigcup_{b_0,\dots,b_{n-1}\in B}\tau_{b_n}\dots\tau_{b_0}Q,$$
with renormalization factor $\frac{|\det R|^n}{N^n}$. Moreover,
$\T^n\mu_0$ is a spectral measure with spectrum $\S^n\bz^d$.

\end{corollary}

\begin{proof}
First we check that the union that gives $\M^n Q$ is disjoint. If
not, then there exist $b_0,\dots,b_{n-1}$ and
$b_0',\dots,b_{n-1}'$ in $B$ and $x,x'\in Q$ such that
$$R^{-n}(x+b_0+Rb_1+\dots+R^{n-1}b_{n-1})=R^{-n}(x'+b_0'+Rb_1'+\dots+R^{n-1}b_{n-1}').$$
but this implies that $x-x'\in\bz^d$ and this is impossible,
unless $x=x'$. Since $(B,L)$ is a Hadamard pair, the points in $B$
are incongruent $\mod R\bz^d$. And therefore we get $b_0=b_0',
\dots, b_{n-1}=b_{n-1}'$. Thus the union is disjoint.

Next we compute the measure $\T^n\mu$. It is enough to take $n=1$,
the general case is analogous. We have for $f$ continuous on
$\br^d$
$$\int f\,d\T\mu_0=\frac{1}{N}\sum_{b\in B}\int_Qf(R^{-1}(x+b))\,dx=\frac{|\det R|}{N}\sum_{b\in B}\int_{\tau_b(Q)}f(y)\,dy.$$
This shows that $\T\mu_0$ is Lebesgue measure on $\M Q$
renormalized by $|\det R|/N$.

The fact that $\S^n\bz^d$ is a spectrum for $\T^n\mu_0$ follows
from Proposition \ref{cor2}.
\end{proof}

\begin{proposition}\label{pr15}
Let $(B,L)$ be a Hadamard pair. Let $\Lambda_0$ be a subset of
$\Pi(B)$ and assume that $\Lambda_0$ is the spectrum of some Borel
probability measure $\mu_0$ on $\br^d$ and
$\S\Lambda_0\subset\Lambda_0$. Then the set
$$\Lambda:=\bigcap_{n\geq0} \S^n\Lambda_0$$
is orthogonal in $L^2(\mu_B)$.

\end{proposition}

\begin{proof}
For $n\in\bn$, let $\mu_n:=\T^n\mu_0$. With Proposition
\ref{cor2}, we have that $\mu_n$ is a spectral measure with
spectrum $\S^n\Lambda_0$.

 Also, from \cite{Hut81} we know that $\mu_n=\T^n\mu_0$ converges weakly to $\mu_B$.
 Take two distinct $\lambda_1,\lambda_2$ in $\Lambda$. Then $\lambda_1,\lambda_2$ are in $\S^n\Lambda_0$ for all $n$, and since $\S^n\Lambda_0$ is a spectrum for $\mu_n$, we have
 $$\int e^{2\pi i(\lambda_1-\lambda_2)\cdot x}\,d\mu_n(x)=0.$$
 Since $\mu_n$ converges to $\mu_B$ weakly, we get that
 $$\int e^{2\pi i(\lambda_1-\lambda_2)\cdot x}\,d\mu_B(x)=0.$$
 This proves that $\Lambda$ is orthogonal in $L^2(\mu_B)$.
\end{proof}

\begin{theorem}\label{th4}
In dimension $d=1$, suppose $(B,L)$ is a Hadamard pair. Then the
measure $\mu_B$ is a spectral measure with spectrum
$$\Lambda=\bigcap_{n\geq 0}\S^n\left(\Pi(B)\right).$$
\end{theorem}

\begin{proof}
We use the result in \cite{DJ06b} that states that $\mu_B$ is a
spectral measure with spectrum $\tilde\Lambda$ - the smallest
$\S$-invariant set that contains $-C$ for all $B$-extreme
$L$-cycles $C$.

\begin{definition}\label{def4}

Define the IFS
$$\tau_l^{(L)}(x)=\Rsi(x+l),\quad(x\in\br^d).$$
To simplify the notation we will use just $\tau_l$ for
$\tau_l^{(L)}$, the subscript $l$ will indicate that we use the
map $\tau_l^{(L)}$ and the subscript $b$ will indicate that we use
the map $\tau_b$.

We say that a finite set $C=\{x_0,x_1,\dots, x_{p-1}\}$ is an
$L$-cycle, if there exists $l_0,l_1,\dots,l_{p-1}\in L$ such that
$$\tau_lx_i=x_{i+1},\quad i\in\{0,\dots,p-1\},$$
where $x_p:=x_0$.

We say that this cycle is $B$-extreme if
$$|m_B(x_i)|=1\mbox{ for all }i\in\{0,\dots,p-1\}.$$

\end{definition}

Thus, we have to prove only that $\Lambda=\cap_n \S^n\Pi(B)$ is
equal to $\tilde\Lambda$.

We need the following lemma:
\begin{lemma}\label{th3}\cite[Theorem 4.1]{DJ07a}
Assume $(B,L)$ is a Hadamard pair.
Suppose there exist $d$ linearly independent vectors in the set
\begin{equation}
\Gamma(B):=\left\{\sum_{k=0}^nR^kb_k : b_k\in B, n\in \bn\right\}.
\label{eq1.15}
\end{equation}
Define
\begin{equation}
\Gamma(B)^\circ:=\left\{ x\in\br^d : \beta\cdot x\in\bz\mbox{ for
all }\beta\in\Gamma(B)\right\}. \label{eq1.16}
\end{equation}
Then $\Gamma(B)^\circ$ is a lattice that contains $\bz^d$ which is
invariant under $\Rs$, and if $l,l'\in L$ with $l-l'\in
\Rs\Gamma(B)^\circ$ then $l=l'$. Moreover
\begin{equation}
\Gamma(B)^\circ\cap X_L\supset\bigcup\left\{ C : C\mbox{ is a
$B$-extreme $L$-cycle}\right\}. \label{eq1.17}
\end{equation}

\end{lemma}

Note first that $\Pi(B)=\Gamma(B)^\circ$, because if
$b\gamma\in\bz$ then $R^kb\gamma\in\bz$ for all $k\geq 0$.

Since $L$ is contained in $\bz$, we have $\S\Gamma(B)^\circ\subset
\Gamma(B)^\circ$ and this implies that $\S\Lambda\subset\Lambda$.
Also, since all $B$-extreme $L$-cycles are contained in
$\Gamma(B)^\circ$, and since $\S(-C)\supset -C$ if $C$ is a cycle,
we obtain that $\Lambda$ contains all these cycles. Therefore
$\tilde \Lambda\subset \Lambda$.

Now take $x_0\in\Lambda$. Then, since $x_0\in\S\Gamma(B)^\circ$,
there is a $x_1\in\Gamma(B)^\circ$ and $l_0\in L$ such that $x_0=R
x_1+l_0$. Since elements in $L$ are incongruent mod
$R\Gamma(B)^\circ$, it follows that $l_0$ is uniquely determined
by $x_0$. So $x_1$ is also uniquely determined; and, since
$x_0\in\Lambda$ it follows that $x_1\in\Lambda$. This implies that
$\tau_{l_0}(-x_0)=-x_1$. By induction we can find an infinite
sequence $l_0,l_1,\dots,l_n\dots$ in $L$ such that
$x_{n+1}=-\tau_{l_n}\dots\tau_{l_0}(-x_0)$ is in $\Lambda$. But
then $x_{n+1}$ converges to the attractor $X_L$ of the IFS
$(\tau_l)_{l\in L}$ and it is also a sequence in
$\Lambda\subset\Gamma(B)^\circ$. Therefore from some point on the
sequence has to be in a cycle inside $X_L\cap\Gamma(B)^\circ$.
Since $|m_B(x)|=1$ for $x\in\Gamma(B)^\circ$, we see that this
implies that $x_n$ will land in one of the $B$-extreme $L$-cycles.

Hence, every point $x_0$ can be obtained from a cycle point after
the application of the operations $x\mapsto Rx+l$ , $l\in L$.
Therefore $\Lambda\subset\tilde\Lambda$.

\end{proof}

\begin{remark}\label{18}
Theorem \ref{th4} is false in higher dimensions. For example, take

$$R=\begin{bmatrix}
2&0\\
1&2\end{bmatrix},\quad B=\left\{ \begin{pmatrix} 0\\0
\end{pmatrix},
\begin{pmatrix}
1\\0
\end{pmatrix}
,
\begin{pmatrix}
0\\3
\end{pmatrix}
,
\begin{pmatrix}
1\\3
\end{pmatrix}
\right\}\quad L=\left\{
\begin{pmatrix}
0\\0
\end{pmatrix}
,
\begin{pmatrix}
1\\0
\end{pmatrix}
,
\begin{pmatrix}
0\\1
\end{pmatrix}
,
\begin{pmatrix}
1\\1
\end{pmatrix}
\right\}$$

\begin{figure}[ht]\label{fig1}
\centerline{ \vbox{\hbox{\epsfxsize 5cm\epsfbox{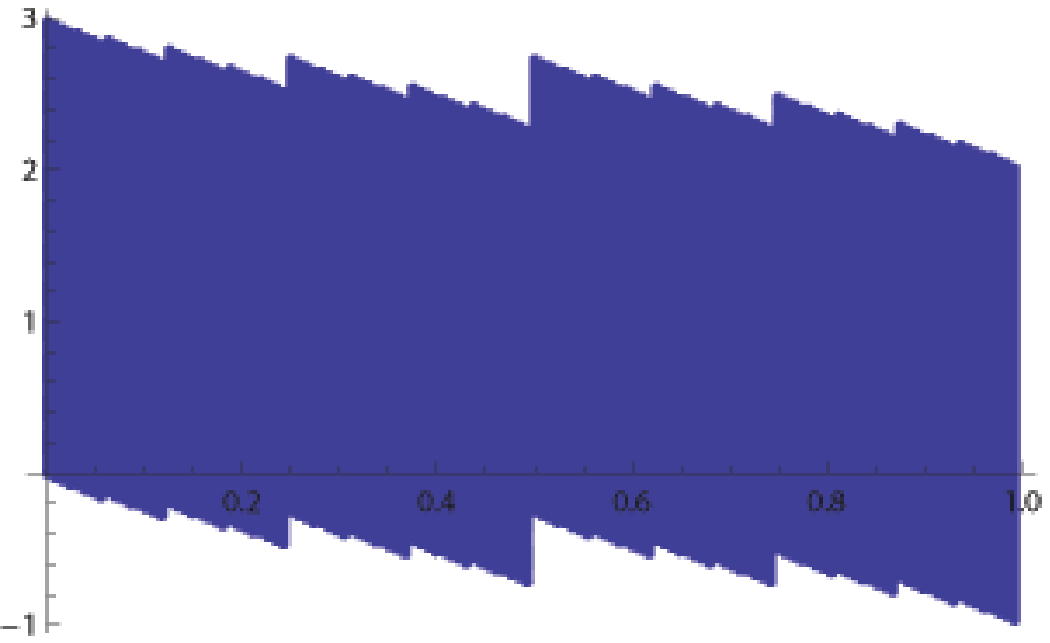}}} }
\caption{ The attractor $X(B)$.}
\end{figure}

We proved in \cite{DuJo10} that the measure $\mu_B$ is the
Lebesgue measure on the attractor $X(B)$ and it has spectrum
$\bz\times\frac13\bz$.

Note that the set $$\Pi(B)=\{\gamma\in\br^2 : \gamma\cdot
b\in\bz\mbox{ for all }b\in B\}$$ is in this case
$\Pi(B)=\bz\times\frac13\bz$.

It is easy to check that $(0,2/3)^T$ is not in $\S\Pi(B)$ and
therefore $\cap_{n\geq 0}\S^n\Pi(B)$ is a proper subset of
$\Pi(B)=\bz\times\frac13\bz$, thus it is incomplete.
\end{remark}

\section{Riesz bases, Bessel sequences and frames}

In Proposition \ref{prgr} and Theorem \ref{thb1} we give
characterization of Bessel spectra, first in terms of the Grammian
matrix, for a general measure, and then in terms of some finite
matrices for affine IFS measures. Theorem \ref{thb1} is then
reformulated in Theorem \ref{thb1re} in terms of an uniform bound
for some finite atomic measures. In Theorem \ref{thsufr} we give a
sufficient condition for a set to be a frame spectrum for an
affine IFS with no overlap, in terms of the same finite matrices.


\begin{proposition}\label{pr1.14}
Suppose the sequence of Borel measures $\mu_{n}$ on $\br^d$
converges to the Borel measure $\mu$ weakly.

\begin{enumerate}
\item Assume $\Lambda_{n}$ is a Riesz basic sequence for $\mu_{n}$
with bounds $A_n$ and $B_n$.  Suppose finally that $\lim A_n = A >
0$, and $\lim B_n = B < \infty$.  Then $\cap \Lambda_{n}$ is a
Riesz basic sequence for $\mu$ (provided the intersection is
nonempty).

\item
 Suppose $\Lambda_{n}$ is a Schauder basic sequence in $L^2(\mu_{n})$ with constant $C_n$.  Suppose $\lim C_n = C < \infty$.  Then $\cap \Lambda_{n}$ is a Schauder basic sequence in $L^2(\mu)$ with constant $C$.
 \end{enumerate}
\end{proposition}

\begin{proof} (i)
Consider a finite set of $K$ frequencies in $\cap \Lambda_{n}$,
say $\lambda_{1}, \dots , \lambda_{K}$.  For the measure
$\mu_{n}$, $\{e_{\lambda_{1}}, \dots, e_{\lambda_{K}} \} \subset
L^2(\mu_n)$ is a Riesz basic sequence with bounds between $A_n$
and $B_n$.  Thus, the Grammian matrix $M_{n}^{K}$ given by
\[ M_{n}^{K}[i,j] = \langle e_{\lambda_{i}}, e_{\lambda_{j}} \rangle_{L^2(\mu_{n})} \]
satisfies
\[ A_n^2 I \leq M_{n}^{K} \leq B_{n}^2. \]
Note that the sequence $M_{n}^{K}$ converges entry-wise to the
Grammian matrix $M^{K}$ given by
\[ M^{K}[i,j] = \langle e_{\lambda_{i}}, e_{\lambda_{j}} \rangle_{L^2(\mu)}. \]
Given $\epsilon > 0$, find $N$ such that the Frobenius norm
$\|M^{K} - M_{n}^{K} \|_{2} < \epsilon$ for $n > N$.  Thus, $M^{K}
- M_{n}^{K}$ is self-adjoint with norm less than $\epsilon$, hence
\[ -\epsilon I \leq M^{K} - M_{n}^{K} \leq \epsilon I . \]
Therefore,
\[
A_n^2I - \epsilon I \leq M_{n}^{K} + M^{K} - M_{n}^{K} \leq B_n^2
I + \epsilon I
\]
and taking limits, we get
\[ A^2 I - \epsilon I \leq M^{K} \leq B^2 I + \epsilon I. \]
This is true for any $\epsilon$. Thus $\{e_{\lambda_{1}}, \dots ,
e_{\lambda_{K}} \} \subset L^2(\mu)$ is a Riesz basic sequence
with bounds between $A$ and $B$.  Since
$\{\lambda_1,\dots,\lambda_K\}$ were arbitrary, this is true for
all of $\Lambda$.

(ii) Let $\| \cdot \|_{n}$ denote the norm in $L^2(\mu_{n})$ and
$\| \cdot \|$ denote the norm in $L^2(\mu)$.  Note that any subset
of $\Lambda_{n}$ is also a Schauder basic sequence in
$L^2(\mu_{n})$ with basis constant no greater than $C_n$.

Fix constants $\{a_{1}, \dots , a_{K}\}$ and $1 \leq k \leq K$,
and let $\{\lambda_{1}, \dots ,\lambda_{K} \}$ be a finite subset
of $\cap \Lambda_{n}$.  We have
\[ \left\| \sum_{j=1}^{k} a_{j} e_{\lambda_{j}} \right\|_{n} \leq C_{n} \left\| \sum_{j=1}^{K} a_{j} e_{\lambda_{j}} \right\|_{n} \]
Taking limits, we get
\[ \left\| \sum_{j=1}^{k} a_{j} e_{\lambda_{j}} \right\| \leq C \left\| \sum_{j=1}^{K} a_{j} e_{\lambda_{j}} \right\|. \]
\end{proof}

Next, we give some characterizations of Bessel spectra. One is in
terms of the Gram matrix (Proposition \ref{prgr}) and it applies
to general Borel measures. The other is in terms of the norm of
some finite matrices (Theorem \ref{thb1}) and applies to affine
IFS measures.

\begin{proposition}\label{prgr}
Let $\mu$ be a Borel probability measure on $\br^d$. Then a
discrete subset $\Lambda$ of $\br^d$ is a Bessel spectrum with
bound $M$ if and only if the matrix
$$(\widehat\mu(\lambda-\lambda'))_{\lambda,\lambda'\in\Lambda}$$
defines a bounded operator on $l^2(\Lambda)$ with norm less than
$M$.
\end{proposition}

\begin{proof}
By \cite[Lemma 3.5.1]{Chr03}, $(e_\lambda)_{\lambda\in\Lambda}$ is
a Bessel sequence with bound $M$ if and only if its Gram matrix
$$(\ip{e_\lambda}{e_{\lambda'}})_{\lambda,\lambda'\in\Lambda}$$
defines a bounded operator on $l^2(\Lambda)$ with norm less than
$M$. But
$\ip{e_\lambda}{e_\lambda'}=\widehat\mu(\lambda-\lambda')$ for all
$\lambda,\lambda'\in\Lambda$.
\end{proof}

Then, an application of Schur's lemma gives the following (see
\cite[Proposition 3.5.4]{Chr03}):
\begin{proposition}\label{prgr2}
Let $\mu$ be a Borel probability measure on $\br^d$. Let $\Lambda$
be discrete subset of $\br^d$. If there exists a constant $M>0$
such that
$$\sum_{\lambda'\in\Lambda}|\widehat\mu(\lambda-\lambda')|\leq M \mbox{ for all }\lambda\in\Lambda,$$
then $\Lambda$ is a Bessel spectrum with Bessel bound $M$.

\end{proposition}

 {\bf Notation.} For $k=(k_1,\dots,k_n)\in B^n$ we define
$$\tau_k:=\tau_{k_n}\circ\dots\circ\tau_{k_1}.$$
\begin{theorem}\label{thb1}
Assume the measure $\mu$ associated to the affine IFS
$(\tau_b)_{b\in B}$  has no overlap. Let $\Lambda$ be a discrete
subset of $\br^d$. Then the following assertions are equivalent:
\begin{enumerate}
\item The set $\Lambda$ is a Bessel spectrum for $\mu$. \item
There exist $r_0>0$ and $C>0$ with the following property: if for
all $n\in\bn$ we define
$$\Lambda_n(r_0):=\{\lambda\in\Lambda :  |\Rs^{-n}\lambda|\leq r_0\}$$
then the norm of the matrix
\begin{equation}
\label{eqb1} \frac{1}{\sqrt{N^n}}\left(e^{-2\pi
i\lambda\cdot\tau_k(0)}\right)_{\lambda\in\Lambda_n(r_0), k\in
B^n}
\end{equation}
is bounded by $C$. \item For all $r_0>0$ there exists a $C(r_0)>0$
such that for all $n\in\bn$ the norm of the matrix in \eqref{eqb1}
is bounded by $C(r_0)$.
\end{enumerate}
\end{theorem}

We begin with some lemmas.
\begin{lemma}\label{lemb1}\cite[Lemma2.2]{DHSW11}
Let $X=X_B$ be the attractor of the IFS $(\tau_b)_{b\in B}$. For
all $n\in\bn$ and $k\in B^n$
$$\int_{\tau_k(X)}f\,d\mu=\frac{1}{N^n}\int_X f\circ\tau_k\,d\mu,\quad(f\in L^\infty(X)).$$
\end{lemma}

\begin{lemma}\label{lemb2}
Let $f=\sum_{k\in  B^n}c_k\chi_{\tau_k(X)}$. Then
\begin{equation}\label{eqb2.1}
\ip{f}{e_\lambda}=\frac{1}{N^n}\widehat\mu(-\Rs^{-n}\lambda)\sum_{k\in
B^n}c_ke^{-2\pi i\lambda\cdot\tau_k(0)},
\end{equation}
\begin{equation}\label{eqb2.2}
\|f\|^2=\frac{1}{N^n}\sum_{k\in B^n}|c_k|^2.
\end{equation}
\end{lemma}
\begin{proof}
We have, using Lemma \ref{lemb1}:
$$E:=\ip{f}{e_\lambda}=\sum_{k\in B^n}c_k\int_{\tau_k(X)}e^{-2\pi i\lambda\cdot x}\,d\mu(x)=
\sum_{k\in  B^n}c_k\frac{1}{N^n}\int_X e^{-2\pi i\lambda\cdot
\tau_k(x)}\,d\mu(x)$$ and since $\tau_k(x)=\tau_k(0)+R^{-n}x$:
$$E=\frac{1}{N^n}\sum_{k\in B^n}c_k e^{-2\pi i\lambda\cdot\tau_k(0)}\int_X e^{-2\pi i\lambda\cdot R^{-n}x}\,d\mu(x)$$
and  \eqref{eqb2.1} follows.

Equation \eqref{eqb2.2} can be obtained from a simple computation
(since $\mu$ has no overlap, the measure of $\tau_k(X)$ is
$1/N^n$).
\end{proof}

\begin{proof}[Proof of Theorem \ref{thb1}]
$(i)\Rightarrow(ii)$. Since $\widehat\mu$ is continuous and
$\widehat\mu(0)=1$, there exist $\delta>0$ and $r_0>0$ such that
$|\widehat\mu(x)|^2\geq\delta$ if $|x|\leq r_0$.

Using Lemma \ref{lemb2} and the Bessel inequality, we have for any
$n$ and any $f$ of the form $f=\sum_{k\in\bn}c_k\chi_{\tau_k(X)}$:
$$M\frac{1}{N^n}\sum_{k\in B^n}|c_k|^2=M\|f\|^2\geq \sum_{\lambda\in\Lambda_n(r_0)}|\ip{f}{e_\lambda}|^2=
\sum_{\lambda\in\Lambda_n(r_0)}\frac{1}{N^{2n}}|\widehat\mu(-\Rs^{-n}\lambda)|^2\left|\sum_{k\in
B^n}c_ke^{-2\pi i \lambda\cdot\tau_k(0)}\right|^2.$$ But, if
$\lambda\in\Lambda_n(r_0)$, then $|\Rs^{-n}\lambda|\leq r_0$ so
$|\widehat\mu(\Rs^{-n}\lambda)|^2\geq\delta$. Therefore
$$M\sum_{k\in B^n}|c_k|^2\geq\delta \sum_{\lambda\in\Lambda_n(r_0)}\left|\frac{1}{\sqrt{N^n}}\sum_{k\in B^n}c_ke^{-2\pi i\lambda\cdot\tau_k(0)}\right|^2$$
which shows that the norm of the matrix in \eqref{eqb1} is less
than $\sqrt{M/\delta}$.

$(ii)\Rightarrow(i)$ It is enough to prove the Bessel bound for
functions of the form $f=\sum_{k\in B^m}c_k\chi_{\tau_k(X)}$
because these are dense in $L^2(\mu)$.

To see that these functions are dense in $L^2(\mu)$, take first a
continuous function $f$ on $X=X_B$ and $\epsilon>0$. Since $X$ is
compact, the function $f$ is uniformly continuous. Take $m$ large
enough such that the diameter of all sets $\tau_k(X)$, $k\in B^m$,
is small enough so that $|f(x)-f(y)|<\epsilon$ for all
$x,y\in\tau_k(X)$ and all $k\in B^m$. Then, using the non-overlap,
define
$$g:=\sum_{k\in B^m}f(\tau_k(0))\chi_{\tau_k(X)}.$$
It is easy to see that $\sup_{x\in X}|f(x)-g(x)|\leq \epsilon$.
This proves that these step functions are dense in $C(X)$, and
since $\mu$ is a regular Borel measure, they are dense in
$L^2(\mu)$.

If $f$ is of this form and  $n\geq m$ then we can say $f$ is also
of the form $f=\sum_{k\in B^n} c_k\chi_{\tau_k(X)}$ because
$(\tau_k(X))_{k\in  B^n}$ is a refinement of the family of sets
$(\tau_k(X))_{k\in B^m}$.

So take $n\geq m$. We have, using Lemma \ref{lemb2},
$$\sum_{\lambda\in\Lambda_n(r_0)}|\ip{f}{e_\lambda}|^2=\frac{1}{N^{2n}}\sum_{\lambda\in\Lambda_n(r_0)}|\widehat\mu(-\Rs^{-n}\lambda)|^2\left|\sum_{k\in B^n}c_ke^{-2\pi i\lambda\cdot\tau_k(0)}\right|^2$$
and, since $|\widehat\mu|\leq 1$, and using the hypothesis, we get
further
$$\leq \frac{1}{N^{2n}}\sum_{\lambda\in\Lambda_n(r_0)}\left|\sum_{k\in B^n}c_ke^{-2\pi i\lambda\cdot\tau_k(0)}\right|^2
\leq C\frac{1}{N^n}\sum_{k\in B^n}|c_k|^2=C\|f\|^2.$$ Since $n$
was arbitrary, we obtain
$$\sum_{\lambda\in\Lambda}|\ip{f}{e_\lambda}|^2\leq C\|f\|^2.$$
This proves (i).

$(iii)\Rightarrow(ii)$ is clear.

$(ii)\Rightarrow(iii)$. Take $r_1>0$. Since $\Rs^{-1}$ is
contractive, there exists $m\in\bn$ with the property that
$\Rs^{-m}B(0,r_1)\subset B(0,r_0)$. Then, for all $n$ we have
$$\Lambda_{n}(r_1):=\{\lambda\in\Lambda : |\Rs^{-n}\lambda|\leq r_1\}\subset
\{\lambda\in\Lambda : |\Rs^{-(n+m)}\lambda|\leq r_0\}.$$ Also,
since $0\in B$, we have
$$\{\tau_k(0) : k\in  B^n\}\subset\{\tau_k(0) : k\in B^{n+m}\}.$$
Then the matrix
$$A(n,r_1):=\left(e^{-2\pi i\lambda\cdot\tau_k(0)}\right)_{\lambda\in\Lambda_n(r_1), k\in  B^n}$$
is a submatrix of the matrix
$$A(n+m,r_0):=\left(e^{-2\pi i\lambda\cdot\tau_k(0)}\right)_{\lambda\in\Lambda_{n+m}(r_0), k\in B^{n+m}}.$$
Therefore, for all $n$
$$\left\|\frac{1}{\sqrt{N^n}}A({n,r_1})\right\|=\sqrt{N^m}\left\|\frac{1}{\sqrt{N^{n+m}}}A({n,r_1})\right\|\leq\sqrt{N^m}\left\|\frac{1}{\sqrt{N^{n+m}}}A({n+m,r_0})\right\|\leq \sqrt{N^m}C(r_0).$$
This proves (iii).
\end{proof}

Conditions (ii) and (iii) in Theorem \ref{thb1} can be expressed
in terms of Bessel sequences for finite atomic measures. We
formulate this in the following lemma:

\begin{lemma}\label{prgrf}
Let $F$ be a finite subset of $\br^d$ and let
$\delta_F:=\frac{1}{\#F}\sum_{f\in F}\delta_f$, where $\delta_f$
is the Dirac measure at $f$. Let $G$ be a finite subset of
$\br^d$. Then $G$ is a Bessel spectrum for $\delta_F$ with bound
$M$ if and only if the matrix
$$\frac{1}{\sqrt{\#F}}\left(e^{2\pi i f\cdot g}\right)_{f\in F,g\in G}$$
has norm less than $\sqrt{M}$.
\end{lemma}

\begin{proof}
Let $A$ be this matrix. Then, for all $g,g'\in G$:
$$(AA^*)_{g,g'}=\frac{1}{\# F}\sum_{f\in F}e^{2\pi i f\cdot(g-g')}=\widehat\delta_F(g-g').$$
So, using Proposition \ref{prgr} and the fact that
$\|A^*A\|=\|A\|^2$, the result follows.

\end{proof}

Therefore, Theorem \ref{thb1} can be reformulated as follows:

\begin{theorem}\label{thb1re}
Using the notations from Theorem \ref{thb1}, the following
assertions are equivalent:
\begin{enumerate}
\item The set $\Lambda$ is a Bessel spectrum for $\mu$. \item
There exist $r_0>0$ and $C>0$ with the property that for all $n$,
the set $\Lambda_n(r_0)$ is a Bessel sequence with bound $C$ for
the atomic measure $\delta_{ B^n}:=\frac{1}{N^n}\sum_{k\in
B^n}\delta_{\tau_k(0)}$. \item For all $r_0>0$ there exists a
constant $C(r_0)>0$ with the property that for all $n$, the set
$\Lambda_n(r_0)$ is a Bessel sequence with bound $C(r_0)$ for the
atomic measure $\delta_{ B^n}:=\frac{1}{N^n}\sum_{k\in
B^n}\delta_{\tau_k(0)}$.
\end{enumerate}
\end{theorem}

In the next theorem we will give a sufficient condition for a set
to be a frame spectra formulated again in terms of the finite
matrices in \eqref{eqb1}.
\begin{theorem}\label{thsufr}
Assume that the measure $\mu$ associated to the affine IFS
$(\tau_b)_{b\in B}$  has no overlap. Let $\Lambda$ be a discrete
subset of $\br^d$. Let $r_0>0$ and $\delta>0$ be such that
$|\widehat\mu(x)|\geq \delta$ for $|x|\leq r_0$. Define
$$\Lambda_n(r_0):=\{\lambda\in\Lambda : |\Rs^{-n}\lambda|\leq r_0\}$$
and define the matrix $A_n$
$$A_n:=\frac{1}{\sqrt{N^n}}\left(e^{-2\pi i\lambda\cdot\tau_k(0)}\right)_{\lambda\in\Lambda_n(r_0), k\in  B^n}.$$
Suppose there exist constants $m,M>0$ such that
$$m\|f\|^2\leq \|A_nf\|^2\leq M\|f\|^2$$
 for all $f\in \bc^{\# B^n}$ and all $n\in\bn$. Then $\Lambda$ is a frame spectrum for $\mu$.
\end{theorem}

\begin{proof}
The upper frame bound follows from Theorem \ref{thb1}. For the
lower frame bound we use the computations in the proof of Theorem
\ref{thb1}(i)$\Rightarrow$(ii).

We have for any $n$ and any $f$ of the form $f=\sum_{k\in
B^n}c_k\chi_{\tau_k(X)}$:

$$\sum_{\lambda\in\Lambda_n(r_0)}|\ip{f}{e_\lambda}|^2\geq\frac{1}{N^n}\delta \sum_{\lambda\in\Lambda_n(r_0)}\left|\frac{1}{\sqrt{N^n}}\sum_{k\in B^n}c_ke^{-2\pi i\lambda\cdot\tau_k(0)}\right|^2.$$
Using the hypothesis this is bigger than
$$\frac{\delta m}{N^n}\sum_{k\in B^n}|c_k|^2=m\delta\|f\|^2.$$
The density of the functions $f$ of this type proves that
$\Lambda$ is a frame.
\end{proof}

\begin{example}

Let $\tau_{0}(x) = \dfrac{x}{3}$ and $\tau_{1}(x) =
\dfrac{x+2}{3}$.  We consider the following sequence of closed
sets:
\[ \Omega_{0} = [0,1]; \qquad \Omega_{n+1} = \tau_{0}(\Omega_{n}) \cup \tau_{1}(\Omega_{n}). \]
Note that $\cap_{n} \Omega_{n} = C_{3}$, the middle third Cantor
set.  We also endow these sets with a measure $\nu_{n}$ where
$\nu_{n}$ is Lebesgue measure restricted to $\Omega_{n}$,
normalized (by $(3/2)^n$) so that $\| \nu_{n} \| = 1$. As in the
proof of Corollary \ref{cor3}, we have $\T^n\nu_0=\nu_n$.
Therefore, this sequence of measures converges weakly to the
measure $\mu_{3}$.  (See e.g. \cite{Hut81})

We consider the IFS
\begin{equation} \label{Eq:dIFS}
 \rho_{0}(x) = 3x; \qquad \rho_{1}(x) = 3x + 1
\end{equation}
acting on $\bz$.  We also consider the sequence of spectra given
as:
\[ \Gamma_{0} = \bz; \qquad \Gamma_{n+1} = \rho_{0}(\Gamma_{n}) \cup \rho_{1}(\Gamma_{n}). \]

\begin{proposition}
The sequence $\Gamma_{n}$ is a Riesz basic spectrum for the
measure $\nu_{n}$, with basis bounds $A_{n}$, $B_{n}$ which
satisfy
\[ A^{n} \leq A_{n} \leq B_{n} \leq B^{n}, \]
where $A=\sqrt{\frac12}$ and $B=\sqrt{\frac32}$ are the basis bounds for the columns of the
matrix
\[ \dfrac{1}{\sqrt{2}} \begin{pmatrix} 1 & 1 \\ 1 & e^{4 \pi i /3} \end{pmatrix}. \]
Moreover, the Riesz basic sequence $\Gamma_{n}$ is complete in
$L^2(\nu_{n})$.
\end{proposition}

\begin{proof}
The basis bounds $A$ and $B$ can be computed by taking the roots of the largest and smallest eigenvalues of $U^*U$ where $U$ is the given matrix.

Suppose $\Gamma_{n-1}$ is a Riesz basic spectrum for $\nu_{n-1}$
with bounds $C$ and $D$.  Let $\{a_{k}\}_{k \in \Gamma_{n}}$ be a
compactly supported sequence.
\begin{align*}
\int_{\Omega_{n}} | \sum_{k \in \Gamma_{n}} a_k e_{k}(x) |^2 d
\nu_{n} &=
  \int_{\tau_{0}(\Omega_{n-1})} | \sum_{k \in \Gamma_{n}} a_k e_{k}(x) |^2 d \nu_{n} +
  \int_{\tau_{1}(\Omega_{n-1})} | \sum_{k \in \Gamma_{n}} a_k e_{k}(x) |^2 d \nu_{n} \\
  &= \int_{\tau_{0}(\Omega_{n-1})} | \sum_{k \in \rho_{0}(\Gamma_{n-1})} a_k e_{k}(x) + \sum_{k \in \rho_{1}(\Gamma_{n-1})} a_k e_{k}(x) |^2 d \nu_{n} \\
  & \qquad + \int_{\tau_{1}(\Omega_{n-1})} | \sum_{k \in \rho_{0}(\Gamma_{n-1})} a_k e_{k}(x) + \sum_{k \in \rho_{1}(\Gamma_{n-1})} a_k e_{k}(x)|^2 d \nu_{n}.
\end{align*}
Let $f(x) = \sum_{k \in \rho_{0}(\Gamma_{n-1})} a_k e_{k}(x)$ and
$g(x) = \sum_{k \in \rho_{1}(\Gamma_{n-1})} a_k e_{k-1}(x)$ so
that both $f(x)$ and $g(x)$ are finite linear combinations of $\{
e_{k} : k \in \rho_{0}(\Gamma_{n-1}) \}$.  Since $\Gamma_{n-1}
\subset \bz$, we have $\rho_{0}(\Gamma_{n-1}) \subset 3 \bz$, so
$f(x + 2/3) = f(x)$ and $g(x + 2/3) = g(x)$.  Continuing with the
above computation:
\begin{align*}
\int_{\Omega_{n}} | \sum_{k \in \Gamma_{n}} a_k e_{k}(x) |^2 d
\nu_{n}
  &= \int_{\tau_{0}(\Omega_{n-1})} | f(x) + e_1(x) g(x) |^2 d \nu_{n} + \int_{\tau_{1}(\Omega_{n-1})} | f(x) + e_1(x)g(x) |^2 d \nu_{n} \\
  &= \int_{\tau_{0}(\Omega_{n-1})} | f(x) + e_1(x) g(x) |^2 + | f(x + \dfrac{2}{3}) + e_1(x + \dfrac{2}{3})g(x + \dfrac{2}{3}) |^2 d \nu_{n} \\
  &= \int_{\tau_{0}(\Omega_{n-1})} | f(x) + e_1(x) g(x) |^2 + | f(x) + e_1(x + \dfrac{2}{3})g(x) |^2 d \nu_{n} \\
  &= \int_{\tau_{0}(\Omega_{n-1})} \biggm\| \begin{pmatrix} 1 & 1 \\ 1 & e_{1}(\frac{2}{3}) \end{pmatrix} \begin{pmatrix} f(x) \\ e_{1}(x) g(x) \end{pmatrix} \biggm\|^2 d \nu_{n}.
\end{align*}
By hypothesis, we have that
\[ 2 A^2( |f(x)|^2 + |e_1(x)g(x)|^2 ) \leq \biggm\| \begin{pmatrix} 1 & 1 \\ 1 & e_{1}(\frac{2}{3}) \end{pmatrix} \begin{pmatrix} f(x) \\ e_{1}(x) g(x) \end{pmatrix} \biggm\|^2 \leq 2 B^2 ( |f(x)|^2 + |e_1(x)g(x)|^2). \]
Therefore,
\begin{equation} \label{E:ab1}
2 A^2( \int_{\tau_{0}(\Omega_{n-1})} |f(x)|^2 + |g(x)|^2  d
\nu_{n}) \leq \int_{\Omega_{n}} | \sum_{k \in \Gamma_{n}} a_k
e_{k}(x) |^2 d \nu_{n} \leq 2 B^2 ( \int_{\tau_{0}(\Omega_{n-1})}
|f(x)|^2 + |g(x)|^2  d \nu_{n}).
\end{equation}
For $f(x) = \sum_{k \in \rho_{0}(\Gamma_{n-1})} a_{k} e_k(x) =
\sum_{k \in \Gamma_{n-1}} a_{\rho_{0}(k)} e_k(x)$, we have
\begin{align}
\int_{\tau_{0}(\Omega_{n-1})} |\sum_{k \in \rho_{0}(\Gamma_{n-1})} a_{k} e_k(x)|^2  d \nu_{n} &= \int_{\tau_{0}(\Omega_{n-1})} |\sum_{k \in \Gamma_{n-1}} a_{\rho_{0}(k)} e_k(\rho_{0}(x))|^2  d \nu_{n} \notag \\
&= \dfrac{1}{2} \int_{\Omega_{n-1}} |\sum_{k \in \Gamma_{n-1}}
a_{\rho_{0}(k)} e_k(x) |^2 d \nu_{n-1}. \label{E:gn1}
\end{align}
Since $\Gamma_{n-1}$ is a Riesz basic spectrum for $\nu_{n-1}$, we
have
\begin{equation}
A_{n-1}^{2} \left( \sum_{k \in \Gamma_{n-1}} |a_{\rho_{0}(k)}|^2
\right) \leq  \int_{\Omega_{n-1}} | f |^2 d\nu_{n-1}  \leq
B_{n-1}^{2} \left( \sum_{k \in \Gamma_{n-1}} |a_{\rho_{0}(k)}|^2
\right) \label{E:abn1}
\end{equation}
Likewise, we have
\begin{equation*}
A_{n-1}^{2} \left( \sum_{k \in \Gamma_{n-1}} |a_{\rho_{1}(k)}|^2
\right) \leq   \int_{\Omega_{n-1}} | g |^2 d\nu_{n-1}  \leq
B_{n-1}^{2} \left( \sum_{k \in \Gamma_{n-1}} |a_{\rho_{1}(k)}|^2
\right)
\end{equation*}
Combining (\ref{E:ab1}), (\ref{E:gn1}), and (\ref{E:abn1}), we
obtain
\begin{align*}
A_{n-1}^2 A^2 ( \sum_{k \in \Gamma_{n}} |a_{k}|^2 ) &= A_{n-1}^2 A^2 ( \sum_{k \in \Gamma_{n-1}} |a_{\rho_{0}(k)}|^2  + \sum_{k \in \Gamma_{n-1}} |a_{\rho_{1}(k)}|^2 ) \\
& \leq A^2 ( \int_{\Omega_{n-1}} |f|^2 + |g|^2 d \nu_{n-1}) \\
& \leq \int_{\Omega_{n}} | \sum_{k \in \Gamma_{n}} a_{k} e_{k}(x) |^2 d \nu_{n} \\
& \leq B^2 ( \int_{\Omega_{n-1}} |f|^2 + |g|^2 d \nu_{n-1}) \\
& \leq B_{n-1} B^2 ( \sum_{k \in \Gamma_{n}} |a_{k}|^2 ).
\end{align*}
This completes the proof that $\Gamma_{n}$ is a Riesz basic
spectrum for $\nu_{n}$.

We now need to prove that $\Gamma_{n}$ is complete in
$L^2(\nu_{n})$.  Again we proceed by induction, and we will show
that $\Gamma_{n}$ has uniformly dense span in $C(\Omega_{n})$.
Let $g \in C(\Omega_{n})$, and define the functions
\[ k_{1}(x) = \alpha g(x) + \beta g(x + 2/3); \qquad k_{2}(x) = e^{-2 \pi i x} (\gamma g(x) + \delta g(x + 2/3)), \]
where
\[ \begin{pmatrix} \alpha & \beta \\ \gamma & \delta \end{pmatrix} = \begin{pmatrix} 1 & 1 \\ 1 & e^{4 \pi i /3} \end{pmatrix}^{-1} =: A^{-1}. \]
Note that $k_1,k_2 \in C(\Omega_{n})$, so for $x \in \Omega_{n-1}$
the functions defined as $h_1(x) := k_{1}(x/3)$ and $h_2(x) :=
k_{2}(x/3)$ are both elements of $C(\Omega_{n-1})$.  By
hypothesis, there exists $l_{1},l_{2} \in span({\Gamma_{n-1}})$
such that for all $x \in \Omega_{n-1}$, $|h_1(x) - l_{1}(x)| <
\epsilon$ and $|h_{2}(x) - l_{2}(x)| < \epsilon$.  Therefore
$l_{1}(3x) \in span(\rho_{0}(\Gamma_{n-1}))$ and $e^{2 \pi i x}
l_{2}(3x) \in span(\rho_{1}(\Gamma_{n-1}))$.  Note that for $q(x)
\in span(\rho_{0}(\Gamma_{n-1}))$, $q(x + 2/3) = q(x)$, and for
$q(x) \in span(\rho_{1}(\Gamma_{n-1}))$, $q(x+2/3) = e^{4 \pi i
/3} q(x)$.

We consider $f(x) = l_{1}(3x) + e^{2\pi i x} l_{2}(3x) \in
span(\Gamma_{n})$.  For any $x \in \tau_{0}(\Omega_{n-1})$,
\begin{align*}
|g(x) - f(x)| + |g(x + 2/3) - f(x + 2/3)| &= \left\| \begin{pmatrix} g(x) \\ g(x + 2/3) \end{pmatrix} - \begin{pmatrix} l_{1}(3x) + e^{2 \pi i x} l_{2}(3x)  \\ l_{1}(3x) + e^{4 \pi i /3} e^{2 \pi i x} l_{2}(3x) \end{pmatrix} \right\|_{1} \\
&= \left\| \begin{pmatrix} g(x) \\ g(x + 2/3) \end{pmatrix} - A \begin{pmatrix} l_{1}(3x) \\ e^{2 \pi i x} l_{2}(3x) \end{pmatrix} \right\|_{1} \\
&= \left\| A \left( A^{-1} \begin{pmatrix} g(x) \\ g(x + 2/3) \end{pmatrix} - \begin{pmatrix} l_{1}(3x) \\ e^{2 \pi i x} l_{2}(3x) \end{pmatrix} \right) \right\|_{1} \\
&\leq \|A\|_{1} \left\| \left( \begin{pmatrix} k_{1}(x) \\ e^{2 \pi i x} k_{2}(x) \end{pmatrix} - \begin{pmatrix} l_{1}(3x) \\ e^{2 \pi i x} l_{2}(3x) \end{pmatrix} \right) \right\|_{1} \\
&= \|A\|_{1} \left\| \left( \begin{pmatrix} h_{1}(3x) \\ e^{2 \pi i x} h_{2}(3x) \end{pmatrix} - \begin{pmatrix} l_{1}(3x) \\ e^{2 \pi i x} l_{2}(3x) \end{pmatrix} \right) \right\|_{1} \\
&= \|A\|_{1} \left( |h_{1}(3x) - l_{1}(3x)| + |h_{2}(3x) - l_{2}(3x)| \right) \\
&< 2 \|A \|_{1} \epsilon.
\end{align*}
Therefore, $g$ is in the uniform closure of $span(\Gamma_{n})$.
\end{proof}
\end{example}

\begin{remark}
Note that $\cap_{n} \Gamma_{n} = \{0,1,3,4,\dots\} = \{
\sum_{k=1}^{K} l_{k} 3^k : l_k = 0,1 \}$.  By \cite{DHSW11}, this
system cannot be a Bessel spectrum for $\mu_{3}$, and so $B_{n}$
diverges to $\infty$.  A natural question follows: is $\cap_{n}
\Gamma_{n}$ a Schauder basic spectrum for $\mu_{3}$?
\end{remark}

\begin{remark}
Note that $\Gamma_{n}$ are determined by the ``dual'' iterated function
system given by $\rho_{0}$ and $\rho_{1}$ in Equation \ref{Eq:dIFS}. Indeed,
\[ \Gamma_{n} = \cup_{k=1}^{2^n} 3^n \bz + q_{k} \]
where $q_{k} \in \{0, 1 , \dots , 3^n - 1\}$, and correspond to the
orbit of $0$ under words of $\rho_{0}$, $\rho_{1}$ of length $n$.

Consider now the possibility of freely choosing the $q_{k}$'s at
each scale $n$:  for each $n$, choose $Q_{n} \subset \{0, 1 , \dots
, 3^n - 1\}$ of size $r_{n}$ (not necessarily $2^n$) and let
\[ \Gamma_{n} = \cup_{k=1}^{r_{n}} 3^n \bz + q_{k}. \]

Let $P_n=\{p_1<p_2<\dots<p_{2^n}\}$ be the integers that appear as left endpoints for the intervals of the set $3^n\Omega_n$. The set $P_n$ can be defined recursively as a set of integers
$P_{n} \subset \{0, 1, \dots , 3^n-1\}$ by $P_{n}
= \cup_{p \in P_{n-1}} \{ p, p + 2 \cdot 3^n \}$, $P_{1} = \{0,
2\}$.
\end{remark}

\begin{theorem}
Suppose $Q_{n}$ is such that the $2^n \times r_{n}$ matrix
\[ M(l,m) = e^{2 \pi i q_{l} p_{m} / 3^n} \]
is bounded in norm by $2^{\frac{n}{2}} L$, for some $L$ independent of $n$.
(For convenience, suppose $p_{k} < p_{k+1}$ and $q_{k} <
q_{k+1}$.)  Then
\[ \Gamma_{n} = \cup_{k=1}^{r_n} 3^n \bz + q_{k} \]
is a Bessel spectrum for $\nu_{n}$ with Bessel bound no greater
than $L^2$ and hence $\cap \Gamma_{n}$ is a Bessel spectrum for
$\mu_{3}$ with Bessel bound no greater than $L^2$.

If in addition, the matrix $M$ is bounded below by $2^{\frac{n}{2}} L'$, for
some $L'$ independent of $n$, then $\Gamma_{n}$ is a Riesz basic
spectrum for $\nu_{n}$ with lower basis bound $L'$, and hence
$\cap \Gamma_{n}$ is a Riesz basic spectrum for $\mu_{3}$.
\end{theorem}

\begin{proof}
Fix $n$ and let $\{c_{\gamma}\}_{\gamma \in \Gamma_{n}}$ be a square summable sequence indexed by $\Gamma_{n}$.  We need to estimate the norm $\| \sum_{\gamma \in \Gamma_{n}} c_{\gamma} e_{\gamma} \|^2_{\nu_{n}}$.  Decompose the sum as follows:
\begin{multline*}
\sum_{\gamma \in \Gamma_{n}} c_{\gamma} e_{\gamma} (x) = \sum_{l = 1}^{r_{n}} \sum_{z \in \bz} c_{3^{n}z + q_{l}} e_{3^{n}z + q_{l}}(x) = \sum_{l = 1}^{r_{n}} \sum_{z \in \bz} d_{3^{n}z}^{l} e_{3^{n}z}(x) e_{q_{l}}(x) = \sum_{l = 1}^{r_{n}} f_{l}(x) e_{q_{l}}(x)
\end{multline*}
where $d_{3^{n}z}^{l} = c_{3^{n}z + q_{l}}$ and $f_{l}(x) = \sum_{z \in \bz} d_{3^{n}z}^{l} e_{3^{n}z}(x)$.  Note that $f_{l}(x + \frac{1}{3^{n}}) = f_{l}(x)$.  We have that
\begin{align*}
\| \sum_{\gamma \in \Gamma_{n}} c_{\gamma} e_{\gamma} \|^2_{\nu_{n}} &=  \int | \sum_{\gamma \in \Gamma_{n}} c_{\gamma} e_{\gamma}(t) |^2 d \nu_{n}(t) \\
&= \frac{3^n}{2^{n}} \sum_{m=1}^{2^{n}} \int_{0}^{\frac{1}{3^n}} | \sum_{\gamma \in \Gamma_{n}} c_{\gamma} e_{\gamma}(t + \frac{p_{m}}{3^n}) |^2 d(t) \\
&= \frac{3^n}{2^{n}} \sum_{m=1}^{2^{n}} \int_{0}^{\frac{1}{3^n}} | \sum_{l = 1}^{r_{n}} f_{l}(x + \frac{p_{m}}{3^n}) e_{q_{l}}(x + \frac{p_{m}}{3^n}) |^2 d(t) \\
&=\frac{3^n}{2^{n}} \sum_{m=1}^{2^{n}} \int_{0}^{\frac{1}{3^n}} | \sum_{l = 1}^{r_{n}} f_{l}(x) e_{q_{l}}(x + \frac{p_{m}}{3^n}) |^2 d(t) \\
&= \frac{3^n}{2^{n}} \int_{0}^{\frac{1}{3^n}} \left\| \begin{pmatrix} \sum_{l = 1}^{r_{n}} f_{l}(x) e_{q_{l}}(x + \frac{p_{1}}{3^n}) \\ \vdots \\ \sum_{l = 1}^{r_{n}} f_{l}(x) e_{q_{l}}(x + \frac{p_{2^n}}{3^n}) \end{pmatrix} \right\|^2  d(t) \\
&= \frac{3^n}{2^{n}} \int_{0}^{\frac{1}{3^n}} \left\| \begin{pmatrix} e_{q_{1}}(\frac{p_{1}}{3^{n}}) & \dots & e_{q_{r_{n}}}(\frac{p_{1}}{3^{n}}) \\
\vdots & \ddots & \vdots \\
e_{q_{1}}(\frac{p_{2^n}}{3^{n}}) & \dots & e_{q_{r_{n}}}(\frac{p_{2^n}}{3^{n}})
\end{pmatrix}
\begin{pmatrix} f_{1}(x) e_{q_{1}}(x) \\ \vdots \\ f_{r_{n}}(x) e_{q_{r_n}}(x) \end{pmatrix} \right\|^2  d(t) \\
& \leq \frac{3^n}{2^{n}} \int_{0}^{\frac{1}{3^n}} 2^n L^2 \left\| \begin{pmatrix} f_{1}(x) e_{q_{1}}(x) \\ \vdots \\ f_{r_{n}}(x) e_{q_{r_n}}(x) \end{pmatrix} \right\|^2  d(t)  \\
&= 3^n L^2 \int_{0}^{\frac{1}{3^n}} \left\| \begin{pmatrix} f_{1}(x) \\ \vdots \\ f_{r_{n}}(x) \end{pmatrix} \right\|^2  d(t) \\
&= 3^n L^2 \left( 3^{-n} \sum_{l=1}^{r_{n}} \sum_{z \in \bz} | d_{3^n z}^{l} |^2 \right) \\
&= L^2 \left( \sum_{\gamma \in \Gamma_{n}} | c_{\gamma} |^2 \right).
\end{align*}

If the matrix is bounded below, the inequality above is reversed, with $L$ replaced by $L'$.


\end{proof}

\bibliographystyle{alpha}
\bibliography{spectral}

\end{document}